\documentclass{cimart}

\usepackage{enumerate}

\allowdisplaybreaks

\title{% Please, capitalize only the first word
Partially alternative  algebras
    }

\author{% Please, use "Firstname Lastname" format, without abreviations
   Tianran  Hua, Ekaterina Napedenina and Marina  Tvalavadze}

\authorinfo[T. Hua]{University of Toronto Mississauga, Mississauga, Canada}{tianran.hua@mail.utoronto.ca}

\authorinfo[E. Napedenina]{Plekhanov Russian University of Economics,  Moscow, Russia}{katernap@gmail.com}

\authorinfo[M. Tvalavadze]{University of Toronto Mississauga, Mississauga, Canada}{marina.tvalavadze@utoronto.ca}

\abstract{%
   In this paper, we introduce a novel generalization of the classical property of algebras known as "being alternative," which we term "partially alternative." This new concept broadens the scope of alternative algebras, offering a fresh perspective on their structural properties.  We showed that partially alternative algebras exist in any even dimension. Then we classified middle $\mathbb C$-associative (noncommutative) algebras satisfying partial alternativity condition. We demonstrated that for any four-dimensional partially alternative real division algebra, one can select a basis that significantly simplifies its multiplication table. Furthermore, we established that every four-dimensional partially alternative real division algebra naturally gives rise to a real Lie algebra, thereby bridging these two important algebraic frameworks. Our work culminates in a description of all Lie algebras arising from such partially alternative algebras. These results extend our understanding of algebraic structures and reveal new connections between different types of algebras. 
    }



\VOLUME{33}
\ISSUE{1}
\YEAR{2025}
\NUMBER{9}
\DOI{https://doi.org/10.46298/cm.15406}

\begin{document}

\section{Introduction}

As is well known, an {\it alternative} algebra is defined by two identities $x(xy) = (xx)y$ and $(yx) x = y (xx)$  for all  $x, y$ in this algebra. This can be interpreted as stating that the associator of this algebra, defined as $(x, y, z)= (xy)z - x(yz)$, is an alternating trilinear mapping.  Notably,  alternativity is a weaker form of associativity; hence, the class of  associative algebras is contained in the class of alternative algebras.
The converse is not true, and the most well-known example of an alternative non-associative algebra is the octonion algebra. 

Over the years, alternative algebras have garnered the attention of many researchers who examined these algebras from different perspectives. 
 In 1957, Bruck and Kleinfeld established that an alternative division ring must necessarily be a division algebra of Cayley-Dickson type (a generalization of the Cayley numbers to an arbitrary field) or an associative division algebra \cite{BK}. There exists a substantial body of literature dedicated to exploring representations, cohomology, deformations, actions, and various properties of semisimple alternative algebras.

Subsequent research began to investigate potential generalizations of alternativity. In \cite{Sh} the author introduced the concept of {\it nearly alternative} algebras.
This is a class of noncommutative Jordan algebras $\mathcal J$
satisfying the identity $( [x, y], z, z ) =0$ for all elements $x, y, z\in \mathcal J$. Shestakov also demonstrated the following result.

\begin{theorem}
If $\mathcal J$ is a simple nearly alternative algebra with an idempotent $e\neq 1$, then $\mathcal J$ is commutative or alternative.
\end{theorem}

In \cite{Nik, Mer}, the authors studied {\it almost} alternative algebras, which were originally introduced by Albert. These are algebras over a field $\mathbb F$ of characteristic not 2  defined by a few specific conditions, one of them being
$$ z(xy) = \alpha (zx)y + \beta (zy)x +\gamma (xz)y + \delta (yz)x + \epsilon y(zx) + \eta x(zy) +\sigma y(xz) + \tau x(yz),$$
where all coefficients belong to $\mathbb F$. Under some technical conditions regarding the coeffitients, this algebra becomes Lie-admissible.

In this paper, we propose a further generalization of the concept of alternativity, which we term {\it partial alternativity}. We present examples of algebras that exhibit partial alternativity without being  alternative, thereby illustrating that this new class of algebras encompasses a broader range than the class of alternative algebras. Subsequently, we examine four-dimensional partially alternative real division algebras and elucidate their direct connection to Lie algebras.

One of the possible directions for studying partially alternative algebras involves examining polynomial equations over these structures. The problem of solving polynomial equations over various algebraic structures, such as fields, matrices, and alternative algebras (including quaternions and Cayley-Dickson algebras), is considered a key problem in mathematics. For instance, recent studies have focused on specific types of polynomials over split octonions, as seen in works like  \cite{CV, CL}, which explore solutions to polynomial equations with scalar coefficients over algebraically closed fields. Additionally, other research, such as  \cite{CV}, delves into the roots of octonion polynomials, offering insights into the behavior of these equations in non-associative settings.

\section{Definitions and preliminary example}

Let $\mathcal A$ denote  a nonassociative algebra over $\mathbb R$. Recall that  $(a, b, c)= (ab)c - a(bc)$ is an $\mathbb R$-trilinear map, called an associator, and it measures the degree of nonassociativity. 

In this section we introduce the  \emph{partial alternativity property}  for a nonassociative algebra. This property can be viewed as a natural  generalization of regular alternativity, defined by the identities:
$(x, x, y)=(y, x, x)=0$ for  all  $x, y\in \mathcal A$.  We note  that the latter identities always imply the \emph{flexibility} identity $(x, y, x)=0$.

%\begin{definition}
%Let $\mathcal A$ be a real  algebra with a unit element 1. Then $\mathcal A$ is called \emph{partially alternative} if for every $x\in \mathcal A$ such that $x^2 = -1$ and any $y\in \mathcal A$  we have that 
%$$ (x, x, y)=(y, x, x)=(x, y, x)= 0.  \eqno(1)$$  
%\end{definition}
%\par\medskip
\begin{definition} \label{ImaginaryUnits} 
Let $\mathcal{A}$ be a real nonassociative algebra with unit element $1$. An element $q \in \mathcal{A}$ is called an \textit{imaginary unit}  if $q^2=-1$.
Denote by $\mathcal{I}_\mathcal{A}$  the set of all imaginary units in $\mathcal{A}$.
\end{definition}

\begin{definition}\label{PA}
Let $\mathcal{A}$ be a real  nonassociative algebra with unit element $1$ and $\mathcal{I}_\mathcal{A}\neq \emptyset$. Then
\begin{enumerate}[(1)]
    \item $\mathcal{A}$ is called \textit{partially left alternative}  if  for all $x\in\mathcal{I}_\mathcal{A}$ and $y\in\mathcal{A}$,  $(x,x,y)=0$;
    \item $\mathcal{A}$ is called \textit{partially flexible} if  for all $x\in\mathcal{I}_\mathcal{A}$ and $y\in\mathcal{A}$, $(x,y,x)=0$;
    \item $\mathcal{A}$ is called \textit{partially right alternative} if for all $x\in\mathcal{I}_\mathcal{A}$ and $y\in\mathcal{A}$, $(y,x,x)=0$.
\end{enumerate}
% (1)  $\mathcal{A}$ is called \textit{partially left alternative}  if  for all $x\in\mathcal{I}_\mathcal{A}$ and $y\in\mathcal{A}$,  $(x,x,y)=0$;\\
% (2) $\mathcal{A}$ is called \textit{partially flexible} if  for all $x\in\mathcal{I}_\mathcal{A}$ and $y\in\mathcal{A}$, $(x,y,x)=0$;\\
% (3) $\mathcal{A}$ is called \textit{partially right alternative} if for all $x\in\mathcal{I}_\mathcal{A}$ and $y\in\mathcal{A}$, $(y,x,x)=0$.

The algebra $\mathcal{A}$ is called \textit{partially alternative}  if it is partially left alternative, flexible, and right alternative.
\end{definition}

Notice that the above definition requires the existence of at least one imaginary unit; otherwise, it would become too general, with no condition to hold. 

By Definition \ref{PA}, if $\mathcal{A}$ is partially alternative and $\mathcal{B}$ is a subalgebra of $\mathcal{A}$ containing the identity and isomorphic to $\mathbb{C}$, then $\mathcal{A}$ can be viewed as a $\mathcal{B}$-bimodule. Specifically, for any $x, x' \in \mathcal{B}$ and $y \in \mathcal{A}$, the associator conditions are satisfied: \[(x, x', y) = (x, y, x') = (y, x, x') = 0.\]

It is immediate that  any alternative algebra with unity is  partially alternative and, therefore,  the class of such alternative algebras is contained in the class of partially alternative algebras. Moreover,  this inclusion is strict. 
We soon  provide evidence for this statement.

Next, we introduce the special linear subspace of $\mathcal A$, named its {\it commutative nucleus} that serves as  a measure of the extent to which the given algebra deviates from being commutative. 

\begin{definition}\label{CommNuc} Let $\mathcal A$ be a nonassociative algebra.  Then  the set \[{\mathcal NC}(\mathcal A) = \{ x\in \mathcal A\,|\,  xy =yx,\, \text{for\,any}\, y\in\mathcal A \}\] is called the \emph{commutative nucleus} of $\mathcal A$. 
\end{definition}

It is evident that  ${\mathcal NC}(\mathcal A)$  is a linear subspace that may or may not be a subalgebra. Besides, if  $\mathcal A$  is a unital algebra, its commutative nuclear contains the unit element $1$ and, therefore,  $\mathbb R 1 \subseteq {\mathcal NC}(\mathcal A)$ implies that   $\dim {\mathcal NC}(\mathcal A) \geq 1$.

In conclusion we note that partially alternative algebras exist in any dimension $n$ provided that $n$ is an even number. This follows from the example below.

\begin{example}\label{PAexample1} Consider $\mathcal A_k = \mathbb C \oplus V_1\oplus \ldots \oplus V_k$ with $k\geq 1$,  $\mathbb C=\text{Span}_{\mathbb R}\{ 1, e_1\}$  with $e^2_1 =-1$ where
$1$ denotes the unity of $\mathcal A_k$. Besides each $V_i = \text{Span}_{\mathbb R}\{v_{i1}, v_{i2}\}$ is a two-dimensional subspace. 

We define multiplication of $\mathcal A_k$  as follows.  We set $e_1 \in  {\mathcal NC}(\mathcal A_k) $ which implies that $e_1$ will commute with any  element from $\mathcal A_k$. Also, we let 
 $$e_1 v_{i1} = v_{i2}\quad \text{and}\quad e_1 v_{i2} = -v_{i1}.$$
Next, set $v_{ij}v_{kl}=0$ for all $(i, j)\neq (k, l)$ and $v^2_{ij} = a_{ij} 1,$ where $a_{ij}$ is a positive real number.

\end{example} 

\begin{proposition} The algebra $\mathcal A_k$ defined in Example \ref{PAexample1} is a partially alternative  (non-alternative) algebra of dimension $2k+2$.
\end{proposition}
\begin{proof}
First, it is clear that $\mathcal A_k$ is not alternative since, for example, 
$$ (v_{11}, v_{11}, v_{12}) = v^2_{11} v_{12} - v_{11}(v_{11}v_{12}) = a_{11}v_{12} - v_{11}\cdot 0 = a_{11}v_{12}\neq 0.$$
Let us now show that $e_1$ satisfies partial alternativity condition (i.e.~the left, right and flexible partial alternativity). Thus, we need to show that 
$$ (e_1, e_1, x)=(e_1, x, e_1)=(x, e_1, e_1) =0,$$ for all $x\in\mathcal A_k$.

Due to linearity in $x$, it suffices to verify the conditions only for basis elements of $\mathcal A_k$. Let $x=v_{ij} \in V_i$, $i\in\{1, \ldots, k\}$ and $j\in \{1,2\}$. Then
\begin{align*}
& (e_1, e_1, v_{i1}) = e^2_1 v_{i1} - e_1(e_1 v_{i1}) = -v_{i1} - e_1v_{i2}=-v_{i1}+v_{i1}=0,\\
& (e_1, e_1, v_{i2}) = e^2_1 v_{i2} - e_1(e_1 v_{i2}) = -v_{i2} - e_1(-v_{i1})=-v_{i2}+v_{i2}=0.
\end{align*}
Hence, $(e_1, e_1, x)= 0$ for all $x\in\mathcal A_k$.  The other two conditions, that is,  $(e_1, x, e_1)=0 $ and $(x, e_1, e_1) =0,$ for all $x\in\mathcal A_k$, follow from the first one and the fact that $e_1 \in {\mathcal NC}(\mathcal A_k) $. 

We next show that the set $\mathcal{I}_{\mathcal{A}_k} = \{-e_1, e_1\}.$ For this, we consider any $x = z_0 + v_1+\ldots+v_k$ in $\mathcal A_k$ such that $x^2=-1,$ $z_0\in \mathbb C$ and $v_i\in V_i$. Put 
$z_0= \alpha 1 + \beta e_1$ and $v_i = \alpha_i v_{i1} + \beta_i v_{i2}$ where $\alpha, \beta, \alpha_i, \beta_i \in \mathbb R$.

First, we assume that $z_0=0$, then $x^2 = v^2_1 + \ldots +v^2_k$ where each $v^2_i = (\alpha^2_i a_{i1} + \beta^2_i a_{i2} )1$, and $ \alpha^2_i a_{i1} + \beta^2_i a_{i2} \geq 0$. Hence, $x^2=-1$ is impossible in this case.

If now $z_0\neq 0$, then since $z_0 \in {\mathcal NC}(\mathcal A_k)$ and $v_iv_j=0$ for $i\neq j$, we have that 
$$ x^2= (z^2_0+v^2_1+\ldots+v^2_k) + 2(z_0v_1+\ldots z_0v_k),$$ where the first term is in $\mathbb C$ and each $z_0v_i$ is in $V_i$. As follows from the assumption $x^2=-1$,
each $z_0v_i=0$. In terms of the coefficients 
$$z_0v_i= (\alpha\alpha_i - \beta\beta_i)v_{i1} + (\alpha_i\beta +\alpha\beta_i)v_{i2} =0.$$ This implies that $\alpha\alpha_i  - \beta \beta_i = 0$ and $\beta\alpha_i + \alpha\beta_i = 0.$ This is a system of two linear 
equations in variables $\alpha_i, \beta_i$ and its determinant is $ \alpha^2+\beta^2\neq 0$ as $z_0\neq 0$.   Hence, $\alpha_i=\beta_i=0$ for each $i$.

It follows that in this case $x^2=-1$ implies that $x=z_0\in \mathbb C$, hence,  $x=\pm e_1$. The proof is complete.
\end{proof}

\section {Middle \texorpdfstring{$\mathbb C$}{C}-associative algebras}

Let $\mathcal M$ be an $n$-dimensional real algebra, and $\mathcal C$ be a subalgebra isomorphic to the complex numbers.  Then  $\mathcal M$ is  {\it middle $\mathbb C$-associative} (with respect to $\mathcal C$) if 
\begin{enumerate}[(1)]
    \item $\mathcal M$ is a $\mathcal C$-bimodule  and
    \item $(xz)y=x(zy)$ for all $x, y \in \mathcal M$ and $z\in \mathcal C$ (the middle $\mathbb C$-associativity condition).
\end{enumerate}

% (1) $\mathcal M$ is a $\mathcal C$-bimodule  and \\

% (2) $(xz)y=x(zy)$ for all $x, y \in \mathcal M$ and $z\in \mathcal C$ (the middle $\mathbb C$-associativity condition).\\

In \cite{AHK04}, the authors provide a concise summary of the results regarding the classification of  four-dimensional $\mathbb C$-associative algebras (left, right or middle $\mathbb C$-associative) obtained in a series of their papers. 
 In particular, they fully classified the \emph{commutative} algebras of this type. For the \emph{noncommutative} case, they only identified the canonical multiplication table.

In what follows, we assume that $\mathcal M$ is a four-dimensional real algebra, and focus on the middle $\mathbb C$-associativity condition.  As shown in \cite{AHK}, if $\mathcal M$ is a  noncommutative middle 
$\mathbb C$-associative algebra, then  there exists a basis $\{ 1, i, j, k\}$  in $\mathcal M$ with the following multiplication table:

\begin{equation}\label{Tn} 
\begin{tabular}{c|cccc} 
          & $\bf 1$            & $\bf i$                                  & $\bf j$                                                 & $\bf k$      \\ 
\hline
$\bf 1$     &  $1$           & $i$                                  &  $j$                                                &   $k$     \\
$\bf i$      &  $i$            & $-1$                                &  $k$                                                &   $-j$ \\
$\bf j$     &  $j$          & $-k$                                 &  $a 1 + b i + c j + d k$                   & $f 1 + g i + hj + e k $   \\ 
$\bf k$     & $k$           &  $j$                              &  $-f 1-g i -h j -e k$                & $a 1 + bi + cj +dk $     \\ 
\end{tabular} \tag{$T_n$}
\end{equation}
\noindent where $a, b, c, d, f, g, h, e \in \mathbb R$. Notice that $\mathcal M$ with multiplication table \eqref{Tn} is middle $\mathbb C$-associative  with respect to  $\mathcal C= \text{Span}_{\mathbb R}\{1, i\}$.

In the commutative case (see \cite[Theorem 1.4]{AHK}), there exists a basis $\{ 1, i, j, k\}$  in $\mathcal M$ with the following multiplication table:

\begin{equation}\label{Tc}    
\begin{tabular}{c|cccc} 
          & $\bf 1$            & $\bf i$                                  & $\bf j$                                                 & $\bf k$      \\ 
\hline
$\bf 1$     &  $1$           & $i$                                  &  $j$                                                 &   $k$     \\
$\bf i$      &  $i$            & $-1$                                &  $k$                                                &   $-j$ \\
$\bf j$     &  $j$          & $k$                                 &  $a 1 + b i $                                        & $f 1 + g i + hj $   \\ 
$\bf k$     & $k$           &  $-j$                              &  $f 1+g i + h j $                                   & $-a 1 - bi $     \\ 
\end{tabular} \tag{$T_c$} 
\end{equation} 
\noindent where $a, b, f, g, h\in \mathbb R$, and $h=0$ or 1.  Notice that $\mathcal M$ with multiplication table \eqref{Tc} is middle $\mathbb C$-associative  with respect to $\mathcal C= \text{Span}_{\mathbb R}\{1, i\}.$

Recall that $\mathcal{M}$ is called \emph{strictly} middle $\mathbb{C}$-associative if it is a middle $\mathbb{C}$-associative algebra that does not satisfy either the left or right $\mathbb{C}$-associativity conditions. Specifically, for all $x, y \in \mathcal{M}$ and $z \in \mathcal{C}$, one of the following does not hold: $(zx)y = z(xy)$ and $(xy)z = x(yz)$.

\begin{proposition} Let $\mathcal M$ be a commutative strictly middle $\mathbb C$-associative algebra.  Then $\mathcal M$ is partially alternative.

\end{proposition}

\begin{proof}
To show that  $\mathcal M$ is partially alternative,  we first assume  that there exists an element $g\in \mathcal M$ but not in $\mathcal C =\text{Span}_{\mathbb R}\{1, i\}$ such that $g^2=-1$ and $\mathcal M$ is a $\mathcal C'$-bimodule
where $\mathcal C' = \text{Span}_{\mathbb R} \{1, g\}$. 

Consider $f = ig$. Then 
$$f^2 = (ig)(ig) = (gi)(ig) = g(i (ig) ) = g (i^2 g) =g(-1\cdot g) = - g^2 = -(-1) = 1.$$

Let us now assume that  $\mathcal E= \{1, i, g, f\}$ form a basis for $\mathcal M$. Then  the multiplication table of $\mathcal M$ with respect to $\mathcal E$ is given by

\begin{equation}\label{eq3}
\begin{tabular}{c|cccc} 
          & $\bf 1$            & $\bf i$                                  & $\bf g$                                                 & $\bf f$      \\ 
\hline
$\bf 1$     &  $1$           & $i$                                  &  $g$                                                &   $f$     \\
$\bf i$      &  $i$            & $-1$                                &  $f$                                                &   $-g $ \\
$\bf g$     &  $g$          & $f$                                 &  $-1$                                               & $-i$   \\ 
$\bf f$     &  $f$           &  $-g$                              &  $-i$                & $1$     \\ 
\end{tabular} 
\end{equation}

Using \eqref{eq3}, one can easily confirm that the middle $\mathbb C$-associativity condition also holds with respect to $\mathcal C'$.
Recall that Lemma 1.2 from \cite{AHK} states that   a commutative strictly middle $\mathbb C$-associative algebra  has exactly {\it one} two-dimensional subalgebra with respect to which it is middle $\mathbb C$-associative. Hence, this case is not possible by the lemma.

If now  $\mathcal E$ is {\it not} a basis for $\mathcal M$, then $f$ can be expressed as a linear combination of $1, i, g$ (as they are linearly independent), then $S=\text{Span}_{\mathbb R}\{1, i, g\}$ is a 3-dimensional associative commutative subalgebra whose basis elements satisfy $i^2=g^2=-1$. 

Let now $f  = ig = \alpha 1 + \beta i + \gamma g$ for
some $\alpha, \beta, \gamma \in\mathbb R$. Multiplying this by $g$ from the right, we obtain
\begin{align*}
 fg & = (ig)g=\alpha g + \beta i g +\gamma g^2,\\
 i g^2& = \alpha g + \beta ig  - \gamma 1,\\
 -i &= \alpha g + \beta f - \gamma 1.   
\end{align*}
The latter implies that $\beta f = \gamma 1 - i - \alpha g$. On the other hand, $\beta f = \alpha \beta 1 + \beta^2 i +\beta \gamma g$. If follows that $\beta^2 = -1$, which is impossible.

This implies that the only elements  that satisfy the condition $x^2=-1$ are $\pm i$. It is straightforward to verify that $\mathcal M$ satisfies partial alternativity condition for $x=\pm i$.
Hence, by Definition \ref{PA},  $\mathcal M$ is partially alternative. The proof is complete. 
\end{proof}

\begin{example} A particular case of $\mathcal M$ with table \eqref{Tn} is an algebra obtained from $\mathcal M$ by setting $c=d=e=h=0$, $f=b$ and $g=-a$. These are real algebras  satisfying the left, right and middle $\mathbb C$-associativity conditions:   $$(z, x, y)=(x, z, y)=(x, y, z)=0$$ for any $z\in \mathcal C, x,y\in\mathcal M$.  In addition, if $b\neq 0$, then 
each such $\mathcal M$ is partially alternative but not alternative as $j^2 j \neq j j^2.$
\end{example}

\begin{example} There are noncommutative middle $\mathbb  C$-associative algebras that are not  partially alternative. For example, if we consider $\mathcal M$ with \eqref{Tn} and $(a, b, c, d) = (-1, 0, 0, 0)$,  $gh\neq 0$, then simple
computations show that $(jk)k \neq j (k^2)$ but $k^2=-1$. This violates partial alternativity condition.
\end{example}
%%%%%%%%%%%%%%%%%%%5
%%%%%%%%%%%%%%%%%%%
%%%%%%%%%%%%%%%%%%%%%%

The remainder of the  section presents a classification of all noncommutative middle $\mathbb C$-associative algebras satisfying the partial alternativity condition. As a starting point, we prove a few auxiliary lemmas.
\begin{lemma}\label{lemma1}
Let $\mathcal M$ be a noncommutative middle $\mathbb C$-associative algebra with multiplication table \eqref{Tn}. If $\mathcal M$ is partially left and right alternative, then  $\mathcal{I}_\mathcal{M}\subseteq \operatorname{Span}\{i,j,k\}$.
\end{lemma}
\begin{proof}
Assume there exists $q=r1+xi+yj+zk$ for $r,x,y,z\in\mathbb R$ such that $q^2=-1$ and $r\neq0$. 

Since $\mathcal{M}$ is left  and right partially alternative, we have that $$q(qi)=(q^2)i=-i=i(q^2)=(iq)q.$$ Using multiplication \eqref{Tn}, we obtain
$$q(qi)=  r' 1 + x' i + y' j + z' k,\quad
\text{and}\quad
(iq)q = r' 1 + x' i + y'' j + z'' k,$$
where 

\begin{equation*}
\begin{aligned}[c]
&r'=-2rx-f(y^2+z^2),\\
&x' = r^2-x^2-g(y^2+z^2),\\
&y' = 2rz-h(y^2+z^2),
\end{aligned}
\qquad
\begin{aligned}[c]
&y''= -2rz - h(y^2+z^2),\\
&z' = -2ry-e(y^2+z^2),\\
&z''= 2ry-e(y^2+z^2).
\end{aligned}
\end{equation*}
It follows that $y'=y''$ and $z'=z''$, and, therefore,  $rz=0$ and $ry=0$. 

By the assumption, $r\neq 0$, we have $y=z=0$. Then $q\in \mathcal C =  \operatorname{Span}\{1, i\}$ satisfying $q^2=-1$. Clearly, $q=\pm i$  which contradicts $r\neq 0$. 
Therefore, $\mathcal{I}_\mathcal{M}\subseteq \operatorname{Span}\{i,j,k\},$ as needed.
\end{proof}
\begin{lemma}\label{lemma2}
Let $\mathcal M$ be as in Lemma \ref{lemma1}.  If $\mathcal{I}_\mathcal{M}\setminus\mathcal{C}\neq\varnothing$, then  $b=c=d=0$ in \eqref{Tn} and 
$\mathcal{I}_{\mathcal{M}}=\{xi+yj+zk\mid -x^2+a(y^2+z^2)=-1\}$.
\end{lemma}

\begin{proof}
Consider $q=xi+yj+zk \in  \mathcal{I}_{\mathcal{M}}$.  Computing $q^2$ in terms of $x, y, z$ and equating to $-1$ we get the following equations:

\begin{equation*}
\begin{aligned}[c]
-&x^2+a(y^2+z^2)=-1\\
& b(y^2+z^2) =0
\end{aligned}
\qquad
\begin{aligned}[c]
& c(y^2+z^2) =0 \\
& d(y^2+z^2) =0
\end{aligned}
\end{equation*}

Since $\mathcal{I}_\mathcal{M} \setminus \mathcal{C}\neq \varnothing$, we can choose $q \in \mathcal{I}_\mathcal{M}\setminus\mathcal{C}$ such that  $y^2 + z^2 \neq 0$. Then it follows 
from the above equations that $b=c=d=0$ and  
$$\mathcal{I}_{\mathcal{M}}=\{xi+yj+zk\mid -x^2+a(y^2+z^2)=-1\}.$$
The proof is complete.
\end{proof}

\begin{example}
Let us consider the following three examples of partially alternative algebras given by specific multiplication tables.
\begin{enumerate}[(1)]
    \item The algebra $\mathcal{M}^+$ is defined by the multiplication table
\begin{center}
$$
\begin{tabular}{ c|c c c c }
 $\cdot $ & $1$ & $i$ & $j$ & $k$ \\ \hline
 $1$ & $1$ & $i$ & $j$ & $k$\\  
 $i$ & $i$ & $-1$ & $k$ & $-j$ \\
 $j$ & $j$ & $-k$ & $1$ & $-i$\\
 $k$ & $k$ & $j$ & $i$ & $1$
\end{tabular}
$$
\end{center}
The imaginary units in $\mathcal{M}^+$ are given by
$$\mathcal{I}_{\mathcal{M}^+}=\{(0,x,y,z)\mid -x^2+y^2+z^2=-1\}$$
and form a \textbf{hyperboloid of two sheets} in $\operatorname{Span}\{i,j,k\}$.
\item The algebra $\mathcal{M}^0$ is defined by the multiplication table
\begin{center}$$
\begin{tabular}{ c|c c c c }
 $\cdot $ & $1$ & $i$ & $j$ & $k$ \\ \hline
 $1$ & $1$ & $i$ & $j$ & $k$\\  
 $i$ & $i$ & $-1$ & $k$ & $-j$ \\
 $j$ & $j$ & $-k$ & $0$ & $0$\\
 $k$ & $k$ & $j$ & $0$ & $0$
\end{tabular}$$
\end{center}
The imaginary units in $\mathcal{M}^0$ are given by
$$\mathcal{I}_{\mathcal{M}^0}=\{(0,x,y,z)\mid x^2=1\}$$
and form two \textbf{real parallel planes} in $\operatorname{Span}\{i,j,k\}$.
\item The quaternion algebra $\mathbb H$ is defined by the standard multiplication table
\begin{center}$$
\begin{tabular}{ c|c c c c }
 $\cdot $ & $1$ & $i$ & $j$ & $k$ \\ \hline
 $1$ & $1$ & $i$ & $j$ & $k$\\  
 $i$ & $i$ & $-1$ & $k$ & $-j$ \\
 $j$ & $j$ & $-k$ & $-1$ & $i$\\
 $k$ & $k$ & $j$ & $-i$ & $-1$
\end{tabular}$$
\end{center}
The imaginary units in $\mathbb{H}$ are given by
$$\mathcal{I}_{\mathbb{H}}=\{(0,x,y,z)\mid x^2+y^2+z^2=1\}$$
and form a \textbf{sphere} in $\operatorname{Span}\{i,j,k\}$.
\end{enumerate}

% (1)  Algebra $\mathcal{M}^+$ is defined by the multiplication table
% \begin{center}
% $$
% \begin{tabular}{ c|c c c c }
%  $\cdot $ & $1$ & $i$ & $j$ & $k$ \\ \hline
%  $1$ & $1$ & $i$ & $j$ & $k$\\  
%  $i$ & $i$ & $-1$ & $k$ & $-j$ \\
%  $j$ & $j$ & $-k$ & $1$ & $-i$\\
%  $k$ & $k$ & $j$ & $i$ & $1$
% \end{tabular}
% $$
% \end{center}
% The imaginary units in $\mathcal{M}^+$ are given by
% $$\mathcal{I}_{\mathcal{M}^+}=\{(0,x,y,z)\mid -x^2+y^2+z^2=-1\}$$
% and form a \textbf{hyperboloid of two sheets} in $\operatorname{Span}\{i,j,k\}$.

% (2)  Algebra $\mathcal{M}^0$ is defined by the multiplication table
% \begin{center}$$
% \begin{tabular}{ c|c c c c }
%  $\cdot $ & $1$ & $i$ & $j$ & $k$ \\ \hline
%  $1$ & $1$ & $i$ & $j$ & $k$\\  
%  $i$ & $i$ & $-1$ & $k$ & $-j$ \\
%  $j$ & $j$ & $-k$ & $0$ & $0$\\
%  $k$ & $k$ & $j$ & $0$ & $0$
% \end{tabular}$$
% \end{center}
% The imaginary units in $\mathcal{M}^0$ are given by
% $$\mathcal{I}_{\mathcal{M}^0}=\{(0,x,y,z)\mid x^2=1\}$$
% and form two \textbf{real parallel planes} in $\operatorname{Span}\{i,j,k\}$.

% (3)  Quaternion algebra $\mathbb H$ is defined by the standard multiplication table
% \begin{center}$$
% \begin{tabular}{ c|c c c c }
%  $\cdot $ & $1$ & $i$ & $j$ & $k$ \\ \hline
%  $1$ & $1$ & $i$ & $j$ & $k$\\  
%  $i$ & $i$ & $-1$ & $k$ & $-j$ \\
%  $j$ & $j$ & $-k$ & $-1$ & $i$\\
%  $k$ & $k$ & $j$ & $-i$ & $-1$
% \end{tabular}$$
% \end{center}
% The imaginary units in $\mathbb{H}$ are given by
% $$\mathcal{I}_{\mathbb{H}}=\{(0,x,y,z)\mid x^2+y^2+z^2=1\}$$
% and form a \textbf{sphere} in $\operatorname{Span}\{i,j,k\}$.

\end{example}

\begin{theorem}
Let $\mathcal M$ be a noncommutative middle $\mathbb C$-associative algebra. Then the following affirmations are equivalent:
\begin{itemize}
    \item[(1)] $\mathcal{M}$ is associative;
    \item[(2)] $\mathcal{M}$ is partially left alternative and right alternative with $\mathcal{I}_\mathcal{M}\setminus\mathcal{C}\neq\varnothing$;
    \item[(3)] $\mathcal{M}$ is isomorphic to one of the aforementioned three algebras: $\mathcal{M}^+$, $\mathcal{M}^0$, and $\mathbb H$.
\end{itemize}
\end{theorem}
\begin{proof}
$(1) \iff (3)$. This is Theorem 1.10 in \cite{AHK}.

%$(1)\implies (2)$. Since $\mathcal{M}$ is associative, it is  left alternative, flexible, and right alternative. By Lemma 1, $\mathcal{I_M}$ is a subset of $\operatorname{Span}\{i,j,k\}$.

$(2)\implies (3)$. By Lemma \ref{lemma2}, $b=c=d=0$ and 
$$\mathcal{I}_{\mathcal{M}}=\{xi+yj+zk\mid -x^2+a(y^2+z^2)=-1\}.$$

If $a>0$, then a change of basis by $1'=1$, $i'=i$, $j'=(1/\sqrt{a})j$, and $k'=(1/\sqrt{a})k$ allows us to assume $a=1$. Similarly, if $a<0$, then a change of basis by $1'=1$, $i'=i$, $j'=(1/\sqrt{-a})j$, and $k'=(1/\sqrt{-a})k$ allows us to assume $a=-1$. 
Consider the following three cases:

 Let $a=1$. It is easy to see that $q=\sqrt{2}i+j\in\mathcal{I}_\mathcal{M}$. Then,
$$-i=(qq)i=q(qi)=(-f)1+(-g-2)i+(-h)j+(-e)k.$$
It follows that $f=0$, $g=-1$, $h=0$, $e=0$. Thus, $\mathcal M\cong \mathcal M^+.$

 Let $a=0$. Choose $q=i+j\in\mathcal{I}_\mathcal{M}$. Then,
$$-i=(qq)i=q(qi)=(-f)1+(-g-1)i+(-h)j+(-e)k.$$
It follows that $f=0$, $g=0$, $h=0$, $e=0$. Thus, $\mathcal M\cong \mathcal M^0.$

 Let $a=-1$. Choose $q=j\in\mathcal{I}_\mathcal{M}$. Then,
$$-i=(qq)i=q(qi)=(-f)1+(-g)i+(-h)j+(-e)k.$$
It follows that $f=0$, $g=1$, $h=0$, $e=0$. In particular, this algebra is isomorphic to the quaternion algebra $\mathbb H$ with the usual basis $\{1,i,j,k\}$.

$(3) \implies (2)$ This follows from the explicit multiplication tables of $\mathcal{M}^+$, $\mathcal{M}^0$, and $\mathbb H.$ Namely, each of these algebras is partially left and right alternative and, in addition, they have imaginary units outside $\mathcal C$.

The proof is complete.
\end{proof}

If we omit the condition $\mathcal{I}_\mathcal{M}\setminus\mathcal{C}\neq\varnothing$ (meaning that $I_M={i,-i}$) from the statement of the theorem, then this case becomes very difficult to classify. The corresponding algebras are no longer isomorphic to any of those listed above, since all three algebras given by the tables have additional imaginary units (as shown). In particular, they are also not associative.

\section {Partially alternative division algebras}

Let $\mathcal{A}$ be an arbitrary non-associative algebra over a field $\mathbb{F}$ whose unit element is denoted by 1. Let  $L_a, R_a: \mathcal{A}\rightarrow \mathcal{A}$ be the linear operators of left and right multiplication defined, respectively, by $L_a(x)=ax,$ $R_a(x)=xa$ for all $x\in\mathcal{A}.$ If $L_a, R_a$ are bijective for every nonzero element $a$ in $\mathcal{A}, $ then $\mathcal{A}$ is said to be a {\em division algebra}. As usual,  $Id_\mathcal{A}$ and $Aut(\mathcal{A})$ stand for the identity operator of $\mathcal A$ and the group of all automorphisms of $\mathcal{A},$ respectively.
\begin{definition}
Let $\mathcal A$ be an algebra over $\mathbb R$, and $f \in Aut(\mathcal{A}).$  If $f\neq Id_{\mathcal A}$ and  $f^2=Id_{\mathcal A}$, then 
$f$ is called a \emph{reflection} of $\mathcal A$. In other words, a reflection  is an automorphism of $\mathcal A$ of order two.

\end{definition} 
%%%%%%%%%%%%%%%%%%%%%%%%%
%%%%%%%%%%%%%%%%%%%%%%%
A key result in the theory of real division algebras identifies that the only possible dimensions for finite-dimensional real division algebras are 1, 2, 4, and 8 (see for example \cite{HKR}). The classification of algebras of dimension 
$d$ has been fully established only for $d=1$ (\cite{Rod}), $d=2$ (\cite{HP, D, AK}), and partially completed for $d\in\{4,8\}.$ Specifically, the classification for dimension 4 was achieved for absolute-valued algebras, 
power-commutative algebras and algebras whose derivation Lie algebra is isomorphic to $\mathfrak{su}(2)$ (see \cite{Rod,Ram,CKMMRR,ED2,BO}). In the present paper, we further explore the four-dimensional case under the assumption that the algebra in question is partially alternative.

One intriguing open question in this field concerns whether every four-dimensional real division algebra $\mathcal{A}$, with a nontrivial automorphism group $Aut(\mathcal{A})$, possesses a reflection. If such a property were proven, it would provide invaluable insights into the underlying algebraic structures of these algebras. Currently, no example exists of a four-dimensional real division algebra whose nontrivial automorphism group lacks a reflection. In the specific case of absolute valued algebras, a positive answer has been established in \cite{DDR}. Therefore, our assumption that reflections exist in $Aut(\mathcal{A})$ appears well-founded.

For a linear mapping $\varphi: \mathcal A\to \mathcal A$,  we denote by $\mathcal E_{\lambda}(\varphi)$  the eigenspace in $\mathcal A$ corresponding to the eigenvalue $\lambda\in \mathbb R$. 
The following useful result is taken from \cite{DDR}:

\begin{lemma}\label{lemma3} Let $\varphi \in Aut(\mathcal{A}).$ Then $\mathcal E_1(\varphi),$ $\mathcal E_1(\varphi)+ \mathcal E_{-1}(\varphi)$ and $\mathcal E_1(\varphi^2)$ are subalgebras of $\mathcal{A}$ satisfying 
\[ \mathcal E_1(\varphi)\subseteq \mathcal  E_1(\varphi)+ \mathcal E_{-1}(\varphi)\subseteq \mathcal E_1(\varphi^2). \]

In addition, if  $\varphi\neq Id_\mathcal{A}$, then the following statements are equivalent

\begin{enumerate}[(1)]
 \item $\varphi$ is a reflection of $\mathcal{A},$ 
\item $\varphi$ is diagonalizable, 
\item $\mathcal{A}=\mathcal E_1(\varphi)\oplus \mathcal E_{-1}(\varphi).$
 \end{enumerate} 
 \end{lemma}

Let $\varphi$ be a reflection of $\mathcal A$. Set $\mathcal B =  \mathcal E_1(\varphi)$ and $\mathcal C = \mathcal E_{-1}(\varphi)$.  By Lemma \ref{lemma3}, we have that  $\mathcal A = \mathcal B \oplus \mathcal C$ where $\mathcal B$ is a subalgebra of $\mathcal A$ and 
\begin{equation}\label{eq1}
    \mathcal B \mathcal C = \mathcal C \mathcal B = \mathcal C,\,\, \mathcal C \mathcal C =\mathcal B.
\end{equation} %\mathcal B \mathcal C = \mathcal C \mathcal B = \mathcal C,\,\, \mathcal C \mathcal C =\mathcal B. \eqno (1)$$

We note that the above equalities follow from Lemmas 1 and 2 from \cite{DDR}.

\begin{lemma} \label{lemma4}
Assume that $\mathcal A$ is a four-dimensional unital real division algebra with a reflection $\varphi$. Let  $\mathcal A = \mathcal B \oplus \mathcal C$ where $\mathcal B$ and $\mathcal C$ are as above. Then both
$\mathcal B$ and  $ \mathcal C$ are two-dimensional subspaces. Moreover, $\mathcal B$ is isomorphic to the algebra of complex numbers.
\end{lemma}
\begin{proof}
 Clearly, $1\in \mathcal B$ since $\varphi(1) = 1$. Hence, $\mathcal B$ is a unital subalgebra of $\mathcal A$ which 
is itself division. In turn, this implies that $\dim(\mathcal B)$ is either 1 or 2 (it cannot be 4 as $\varphi$ is a reflection). In particular, there exists a
 nonzero $a\in \mathcal C$ such that the corresponding left-multiplication operator $L_a: \mathcal A \to \mathcal A$ satisfies $L_a(\mathcal B)\subseteq \mathcal C$ and $L_a(\mathcal C)\subseteq \mathcal B$.
Since $L_a$ is bijective, we conclude that $\dim (\mathcal B) = \dim (\mathcal C) = \frac{1}{2} \dim (\mathcal A) = 2$.  

% If $\dim (\mathcal B) = 1$, then $\dim(\mathcal C) = 3$.   In particular, we can choose two linearly independent elements from $\mathcal C$, say, 
%$w_1$ and $w_2$. Since $\mathcal C \mathcal C =\mathcal B,$ we have that $w_1^2$ and $w_1w_2$ are in $\mathcal B$. By assumption, they must be linearly dependent, and, therefore, there exist $\alpha, \beta \in\mathbb R$ (not both zero) such 
%that $\alpha w_1^2 + \beta w_1 w_2 = 0$. Putting $w_1$ outside the parenthesis, we get $w_1(\alpha w_1 + \beta w_2)=0$. Recall that $\mathcal A$ is division, and since $w_1\neq 0$, we get $\alpha w_1 +\beta w_2 =0$, which contradicts to the assumption of independence. 
% Hence, $\dim (\mathcal B)=\dim (\mathcal C) = 2$. 
It is  well known  that a two-dimensional unital real division algebra is necessarily isomorphic to the algebra of complex numbers $\mathbb C$. Hence, 
 $\mathcal B \cong \mathbb C$.
\end{proof}
In what follows we assume that $\mathcal A$ is a  four-dimensional unital real division algebra  $\mathcal A$ with a reflection $\varphi$.  As shown above, $\mathcal A =\mathcal B\,\oplus\, \mathcal C$ where   $\mathcal B$ isomorphic to $\mathbb C$.
Hence, we can choose a basis $\{ 1, i, w, v\}$ for $\mathcal A$ as follows:  $\mathcal B =\text{Span}_{\mathbb R}\{1, i\}$, $\mathcal C=\text{Span}_{\mathbb R}\{w, v\}$,  $i^2=-1$ and $v=wi$.  

\begin{lemma}\label{lemma5} Assume that $\mathcal A$ has a basis $\{1, i, w, v\}$  as above. In addition, we assume that $\mathcal A$ is partially alternative.
Then  either $iw= wi$ or $iw=-wi$.

\end{lemma}

\begin{proof}
By \eqref{eq1},  $\mathcal B \mathcal C = \mathcal C$, and hence we have that $iw=\alpha w +\beta wi$ where $\alpha, \beta\in \mathbb R$.  Since $\mathcal A$ is partially alternative,  $\mathcal A$ is   a $\mathcal B$-bimodule. Therefore,
\begin{align*}
i(wi)=(iw)i &= (\alpha w +\beta wi)i = \alpha wi + \beta (wi) i =\alpha wi + \beta wi^2  =\alpha wi -\beta w.
\end{align*}
Hence, $i(wi) = -\beta w +\alpha wi.$ However, 
\begin{align*}
-w &= i^2 w = i(iw) = i(\alpha w +\beta wi) =\alpha iw + \beta i(wi) = \alpha iw + \beta (-\beta w +\alpha wi)\\
&=\alpha (\alpha w +\beta wi) -\beta^2 w +\alpha\beta wi = \alpha^2 w +\alpha\beta wi -\beta^2 w +\alpha\beta wi\\
&= (\alpha^2 -\beta^2)w + 2\alpha\beta wi.
\end{align*}
We have that $\alpha^2 -\beta^2 = -1$ and $\alpha\beta =0$. Hence, $\alpha =0$ and $\beta =\pm 1$ which implies that $iw=\pm wi$, as needed.
\end{proof}

\begin{proposition} Let $\mathcal A$ be a four-dimensional  real division algebra with unity $1$. Assume that $\mathcal A$ admits  a reflection.  Then ${\mathcal NC}(\mathcal A)= \mathbb R 1$.
\end{proposition}

\begin{proof}
Since $1\in  {\mathcal NC}(\mathcal A), $  $\dim\, {\mathcal NC}(\mathcal A) \geq 1$. Let us first show that  $ {\mathcal NC}(\mathcal A)$ is $\text{Aut}(\mathcal A)$-invariant.  Indeed, let $f\in \text{Aut}(\mathcal A)$, and choose any $x \in   {\mathcal NC}(\mathcal A)$, 
$y\in \mathcal A$. Since $f$ is bijective, there is $y_0\in \mathcal A$ such that $y=f(y_0)$. Hence, 
\[f(x)y=f(x)f(y_0)=f(xy_0)=f(y_0x)=f(y_0)f(x)=yf(x).\] 
Therefore, $f(x)y =yf(x)$ for any $y\in \mathcal A$. Hence, $f(x)\in { \mathcal NC}(\mathcal A)$ for  any $x\in {\mathcal NC}(\mathcal A)$, as required.

Next, by Lemma \ref{lemma3},  $\mathcal A= \mathcal E_1(\varphi) \oplus \mathcal E_{-1}(\varphi)$.  Moreover, $\mathcal E_1(\varphi)$ is a unital subalgebra, and 
$$\mathcal E_1(\varphi) \mathcal E_{-1}(\varphi)= \mathcal E_{-1}(\varphi)\mathcal E_1(\varphi) = \mathcal E_{-1}(\varphi),\qquad \mathcal E_{-1}(\varphi) \mathcal E_{-1}(\varphi) = \mathcal E_1(\varphi).$$   
By Lemma \ref{lemma4} we have that $\dim \mathcal E_1(\varphi) = \dim \mathcal E_{-1}(\varphi) = 2$, and $\mathcal E_1(\varphi) \cong \mathbb C$. 

As  ${\mathcal NC}(\mathcal A)$ is $\text{Aut}(\mathcal A)$-invariant, we have that $\varphi({\mathcal NC}(\mathcal A) )= {\mathcal NC}(\mathcal A)$. Thus, the restriction $\psi = \varphi|_{ {\mathcal NC}(\mathcal A)}$ is a well-defined
linear mapping of ${\mathcal NC}(\mathcal A)$.
Since $\varphi^2 = Id$,  $\psi^2=\varphi^2|_{ {\mathcal NC}(\mathcal A)} = Id$, and, hence, $\psi$ is  an involution. This implies that its Jordan form contains only $1\times 1$ blocks with either 1 or -1 on the main diagonal.   Therefore, 
$\psi$ is diagonalizable on ${\mathcal NC}(\mathcal A)$.

We next  show that $\dim {\mathcal NC}(\mathcal A) = 1$ by eliminating the other possibilities for its dimension.

\noindent {\bf Case 1.} Let us  assume  that $\dim {\mathcal NC}(\mathcal A)$ is 3. Hence, $  {\mathcal NC}(\mathcal A) = \text{Span}_{\mathbb R}\{1, e_1, e_2\}$ where $e_i$ is an eigenvector of $\varphi$ corresponding to $\lambda_i \in\{1, -1\}$.

If both $e_1$ and $e_2$ are in $\mathcal E_1(\varphi)$, then $\dim \mathcal E_1(\varphi) \geq 3$ as $1\in \mathcal E_1(\varphi)$, a contradiction. Therefore, one of $e_1$, $e_2$ must be in $\mathcal E_{-1}(\varphi)$. Without loss of generality, let us assume that 
$e_2\in \mathcal E_{-1}(\varphi)$. Then $e^2_2=z\in \mathcal E_1(\varphi)$. There exists $z_0\in \mathcal E_1(\varphi)\cong \mathbb C$ such that $z^2_0=z$. Since $e_2\in {\mathcal NC}(\mathcal A)$, $e_2$ and $z_0$ commute, we have that
$$(e_2 - z_0)(e_2 + z_0) =e^2_2 - z_0 e_2 +e_2 z_0 - z^2_0 = e^2_2 - z^2_0= z-z=0.$$

However, since $e_2 - z_0 \ne 0,$ $e_2 + z_0 \ne 0$, and $\mathcal A$ is a division algebra, this is not possible.

\noindent {\bf Case 2.} Assume now that $\dim {\mathcal NC}(\mathcal A)$ is 2. Then ${\mathcal NC}(\mathcal A) =\text{Span}_{\mathbb R}\{1, e\}$ where $e$ is an eigenvector of $\psi$. Thus, $e\in \mathcal E_1(\varphi)$ or $e\in \mathcal E_{-1}(\varphi)$. 
If $e\in \mathcal E_1(\varphi)$, then ${\mathcal NC}(\mathcal A) = \mathcal E_1(\varphi )$ by dimension argument. Choose any $w\in \mathcal E_{-1}(\varphi)$. Then $w^2=z\in \mathcal E_1(\varphi)$. There exists $z_0\in \mathcal E_1(\varphi)={\mathcal NC}(\mathcal A)$ 
such that $z^2_0 = z$. Also, $w$ and $z_0$ commute. Then \[(w-z_0)(w+z_0) = w^2 -z_0w+wz_0 -z^2_0 = 0.\]
Like in the previous case,  this is a contradiction.

Once again, if $e\in \mathcal E_{-1}(\varphi),$  then $e^2 = z\in \mathcal E_1(\varphi)$. We can find $z_0 \in \mathcal E_1(\varphi)$ such that $z^2_0 = z$. Since $e\in {\mathcal NC}(\mathcal A)$, we have that 
\[(e-z_0)(e+z_0) =e^2 +ez_0-z_0e -z^2_0=0,\] 
which is  impossible.

This shows that $\dim\,{\mathcal NC}(\mathcal A)=1$, and ${\mathcal NC}(\mathcal A) =\mathbb R 1$.
\end{proof} 
 
\begin{corollary}  Let $\mathcal A$ be a four-dimensional partially alternative real division algebra with unit element $1$. Further assume that $\mathcal A$ admits a reflection $\varphi$ and write  $\mathcal A = \mathcal B \oplus \mathcal C$ where $\mathcal B=  \text{Span}_{\mathbb R}\{1, i\} \cong \mathbb C$.  Then    $i y = - y i$ for all $y\in \mathcal C$. 
\end{corollary}

\begin{proof}
Assume that there is an element $w\in \mathcal C$ such that $iw \neq -wi$.  Then $\{1, i, w, v \}$ where $v=wi$ is a basis of $\mathcal A$ satisfying conditions of  Lemma \ref{lemma5}. Hence, by assumption and using Lemma \ref{lemma5},  we have that $iw= wi$.  Then $iv= i(wi)= (iw)i=(wi)i=vi$. This 
means that $i$ commutes with any element from $\mathcal C$ as $\{w, v\}$ is a basis for it. Since $i$ obviously commutes 
with any element from $\mathcal B$,  we have that $i\in {\mathcal NC}(\mathcal A)$ and $\dim\,{\mathcal NC}(\mathcal A) \ge 2$, which is impossible by Proposition 2.
\end{proof} 

\begin{corollary}\label{cor2} Let $\mathcal A$ be a four-dimensional partially alternative real division algebra with unit $1$.  Assume that $\mathcal A$ admits a reflection $\varphi$.  Then there exists a basis $\{1, i, w, v\}$ of $\mathcal A$ with 
the following multiplication table $(T_p)$:

\begin{equation}   \label{Tp}
\begin{tabular}{c|cccc} 
          & $\bf 1$            & $\bf i$                                  & $\bf w$                                                 & $\bf v$      \\ 
\hline
$\bf 1$     &  $1$           & $i$                                  &  $w$                                                &   $v$     \\
$\bf i$      &  $i$            & $-1$                                &  $-v$                                                &   $w $ \\
$\bf w$     &  $w$          & $v$                                 &  $\alpha_1 {1} +\alpha_2 { i}$      & $\beta_1 {1} +\beta_2 {i}$   \\ 
$\bf v$     &  $v$           &  $-w$                              &  $\delta_1 {1} +\delta_2 {i }$       & $\gamma_1 { 1} + \gamma_2 {i}$  \\ 
\end{tabular} \tag{$T_p$}
\end{equation}  
where $\alpha_1, \alpha_2, \beta_1, \beta_2, \delta_1, \delta_2, \gamma_1, \gamma_2 \in\mathbb R$.

\end{corollary}

\section{ Lie algebras associated with partially alternative algebras}
This section opens with a lemma illustrating the close connection between partially alternative division algebras and Lie algebras.
\begin{lemma}
Let $\mathcal A$ be a four-dimensional unital partially alternative division algebra with a  reflection.  Then $\mathcal A$ becomes  a Lie algebra $\mathcal L$ with respect to the product $[x, y] =xy-yx$ where $x, y\in \mathcal A$.

\end{lemma}

\begin{proof} By Corollary \ref{cor2}, we can choose a basis  $\{1, i, w, v \}$ of $\mathcal A$ with multiplication table \eqref{Tp}.
Since anti-commutativity of the product $[x, y]$ is known, we only verify the \emph{Jacobi identity}:
 $$J(x, y, z)=[x, [y, z]] + [y, [z, x]] + [z, [x, y]] = 0.$$
It sufficies to show that Jacobi identity holds  for any choice of $x, y, z \in \{1, i, v, w\}$.

First,  assume that at least one of $x, y, z$, say $x$, equals to 1.  Then
$J(1, y, z) = 0$ since $[1, a] =0$ for any $a\in \mathcal A$. 

 We now assume that $x, y, z\in \{i, v, w\}$ and  are distinct; otherwise, if, for example, $x=y$, then, clearly, $J(x, x, z)=0.$ 
By setting $x=i$, $y=v$ and $z=w$, we obtain
\begin{align*}
& [ i, [v, w]] + [ v, [w, i] ] + [w, [i, v ] ] = 0 + [ v, -2iw] + [w, -2vi] \\
&= [v, 2v] + [w, 2w] = 2[v, v] +2 [w, w] =0
\end{align*}
\noindent as $[v, w]$ is in $\text{Span}_{\mathbb R}\{1, i\} $, and  $[w, i] = -2iw=2v$, $[v, i] = 2vi =-2w$.
This proves the claim.
\end{proof}

Mubarakzyanov's classification of low-dimensional real Lie algebras, published in 1963 \cite{Mub}, provides a comprehensive framework for understanding the structure of these algebras up to dimension five. This work complements earlier classifications and has been influential in the study of solvable and indecomposable Lie algebras. 

 Building on this foundational work, we will explore how four-dimensional partially alternative real division algebras yield various types of Lie algebras. 

Let us now recall the canonical multiplication tables of relevant Lie algebras, such as $\mathfrak g_{3, 5}, \mathfrak g_{3,7}$, and $\mathfrak g_{4,9}$, which play an important role in our classification:
\begin{enumerate}

\item  $\mathfrak g_{3, 5}= \text{Span}_{\mathbb R}\{e_1, e_2, e_3\}$ where $ [e_1, e_3] = \beta' e_1 - e_2$, $[e_2, e_3] = e_1 + \beta' e_2$ where $\beta' \geq 0$.  This is a solvable Lie algebra.

\item $ \mathfrak g_{3, 7} =  \text{Span}_{\mathbb R}\{e_1, e_2, e_3\}$ where $[e_2, e_3]=e_1,$ $[e_3, e_1]=e_2$ and $[e_1, e_2]=e_3$. This is a simple Lie algebra isomorphic to $\mathfrak{so}(3)$.

\item $\mathfrak g_{4,9}= \text{Span}_{\mathbb R}\{e_1, e_2, e_3, e_4\}$ where $[e_2, e_3] = e_1,$ $[e_1, e_4] = 2\alpha' e_1$, $[e_2, e_4]=\alpha' e_2 - e_3$, $[e_3, e_4] = e_2 +\alpha' e_3.$ This is an indecomposable solvable Lie algebra.

\end{enumerate}

In what follows  $\mathfrak{g}_1$ denotes a 1-dimensional  real Lie algebra.

\begin{proposition}  Let $\mathcal A$ be a four-dimensional partially alternative unital real division algebra with a reflection.
Let $\mathcal L = (\mathcal A,\, [\,\,,\,\,])$ be the Lie algebra associated with $\mathcal A$. Let $[v, w]= \alpha 1+\beta i$ where
$\alpha, \beta \in  \mathbb R$.  Write $\mathcal L = \mathbb R 1 + \mathcal I$ where $\mathcal I = \text{Span}_{\mathbb R}\{i, w, v\}$.
\begin{enumerate}
\item If $\alpha=0$ and $\beta \neq 0$, then $\mathcal I$ is a simple Lie ideal isomorphic to $\mathfrak{so}(3)$. Hence,   $\mathcal L \cong \mathfrak{g}_1 \oplus \mathfrak g_{3,7}.$ 
\item If $\alpha=\beta=0$, then $\mathcal L\cong \mathfrak g_1 \oplus \mathfrak g_{3, 5}.$ 
\item If $\alpha \neq 0$ and $\beta \neq 0$, then  $\mathcal L  \cong \mathfrak{g}_1 \oplus \mathfrak g_{3,7}.$ 
\item  If $\alpha\neq 0$ and $\beta = 0$, then $\mathcal L \cong \mathfrak g_{4,9}$ (with zero parameter).
\end{enumerate}

\end{proposition}

\begin{proof}
We first assume that $\alpha =0$. Then $[v, w] = \beta i \in \mathcal I$. Hence, $\mathcal I$ is a Lie subalgebra which is also a Lie ideal as $[\mathbb R 1, \mathcal I] =\{0\}$. Thus, 
$\mathcal L = \mathbb R 1\oplus \mathcal I$ where $\mathbb R 1, \mathcal I $ are ideals of $\mathcal L$. In particular,  $\mathcal L$ is a decomposable Lie algebra.

In addition, let us assume that $\beta \neq 0$. Then it can be easily seen that the derived algebra
$\mathcal I^{(1)} = [\mathcal I, \mathcal I] = \text{Span}_{\mathbb R}\{i, v, w\} = \mathcal I$. Hence, $\mathcal I$ is a nonsolvable three-dimensional Lie algebra.

 Recall that  $[v, w]=\beta i,\, [i, v]=2w,\, [i, w]=-2v.$  If $\beta >0$,  then by scaling the basis of $\mathcal I$ according to \[h=\dfrac{1}{2} i, \quad e=\dfrac{1}{\sqrt{2\beta}} v, \quad f = \dfrac{1}{\sqrt{2\beta}} w, \] we obtain 
$[e, f]=h,$ $[f, h]=e$ and $[h, e]=f$ which yields  the canonical basis for $\mathfrak {so}(3)$.    If $\beta < 0$,  then we scale the basis of $\mathcal I$ as follows:  
\[h=\dfrac{1}{2} i, \quad e=\dfrac{1}{\sqrt{-2\beta}} v, \quad f = \dfrac{1}{\sqrt{-2\beta}} (-w),\] which results in the canonical basis for $\mathfrak {so}(3).$  Thus,   $\mathcal L =\mathbb R 1 \oplus \mathfrak {so}(3)  \cong \mathfrak{g}_1 \oplus \mathfrak g_{3,7}$.

If now $\alpha=\beta = 0$, then $[v, w]=0$. Hence, $\mathcal I^{(1)}=\text{Span}_{\mathbb R}\{v, w\},$ $\mathcal I^{(2)}=\{0\}$. 
Referring to Mubarakzyanov's classification of 3-dimensional Lie algebras, $\mathcal I$ is isomorphic to  
$\mathfrak g_{3, 2}$, $\mathfrak g_{3, 3}$, $\mathfrak g_{3, 4}$ or $\mathfrak g_{3, 5}$. Moreover, each of these algebras
  has a canonical basis $\{e_1, e_2, e_3\}$ such that the derived Lie algebra is spanned by $\{e_1, e_2\}$, and 
$\text{ad}(e_3)$ can be restricted to it.  If we write $i=\mu e_3 + v_0$ where $v_0 \in \mathcal I^{(1)}$, $\mu\neq 0$, then \[\text{ad}(i)(x) = \text{ad}(\mu e_3)(x) +\text{ad}(v_0)(x)=
\mu \text{ad}(e_3)(x), \]
as $[v_0, x]=0$ for every $x\in \mathcal I^{(1)}$.

Next, we note that $\mathcal I = \mathbb R 1 + \mathcal I^{(1)}$ and $\text{ad}(i)|_{\mathcal I^{(1)}}$ has no real eigenvalues as its characteristic polynomial is
$p(\lambda) =\lambda^2 + 4$.  However, as follows from the multiplication tables of  $\mathfrak g_{3, 2}$, $\mathfrak g_{3, 3}$, $\mathfrak g_{3, 4}$ with respect to $\{e_1, e_2, e_3\}$, 
$\text{ad}(e_3)|_{\mathcal I^{(1)}}$ must have real eigenvalues which is a contradiction. Thus, $\mathcal I$ must be of the remaining type $\mathfrak g_{3, 5}$. Therefore, 
$\mathcal L\cong \mathfrak g_1 \oplus \mathfrak g_{3, 5}.$

We now assume that $\alpha \neq 0$, and $\beta \ne 0$. Thus, $\mathcal L^{(1)}=\text{Span}_{\mathbb R}\{v, w,  \alpha 1+\beta i\}$. Let $e$ denote $\alpha 1+\beta i$. It is easy to see that $ \mathcal L^{(1)}=
\mathcal L^{(2)}$ which implies that $\mathcal L$ is unsolvable. In addition, $[v, w]=e,$ $[e, v] =2\beta w$, $[e, w] = -2\beta v.$ By scaling the basis accordingly, we obtain a new basis with multiplication similar to that of $\mathcal I$ in the first case (i.e. $\alpha =0$, $\beta \neq 0$). It follows that $\mathcal L = \mathbb R1\oplus  \mathcal L^{(1)}$ where $L^{(1)}\cong \mathfrak{so}(3).$

Let us now consider the remaining case when $\beta =0$. This case seems tedious, as we have to rule out several possibilities. We first note that in this case 
$\mathcal L^{(1)}=\text{Span}_{\mathbb R}\{1, v, w\}$, $\mathcal L^{(2)}= \mathbb R 1$, $\mathcal L^{(3)}=\{0\}$. Then $\mathcal L$ is solvable (non-nilpotent) Lie algebra. 

Moreover, $\mathcal L$ is indecomposable as any nonzero ideal of $\mathcal L$ contains 1. Indeed, let $I'\neq \{0\}$ be an ideal of $\mathcal L$, and let $x=\alpha' 1 + \beta' i + \gamma' v + \delta' w \in \mathcal I'$ such that
$\beta'^2 + \gamma'^2 +\delta'^2 \neq 0$. If $\beta'\neq 0$, then $[[x, v], v] = (-2\alpha\beta')1 \neq 0$ and is in $\mathcal I'$. If  $\gamma'$ (or $\delta'$) is nonzero, then 
$[[x, i], v] = (2\alpha\gamma')1 \neq 0$ and is in $\mathcal I',$ as needed. 

According to Mubarakzayanov's classification, $\mathcal L$ is one of the following types: $\mathfrak g_{4, r}$ where $r=2, \ldots, 10$. Since  $\mathcal L^{(2)}= \mathbb R 1$, 
$\mathcal L$ cannot be of types $\mathfrak g_ {4, 2}, \mathfrak g_{ 4, 3}, \mathfrak g_{4, 4}, \mathfrak g_{4, 5}, \mathfrak g_{4, 6}$ or $\mathfrak g_{4, 10}$ for which 
$\mathcal L^{(2)}=\{0\}$.

For the remaining three types: $\mathfrak g_{4, 7}$, $\mathfrak g_{4, 8}$, $\mathfrak g_{4, 9}$ we verify the existence of an element $x\neq 0$ such that $\text{ad}(x)(\mathcal L)=\{0\}$ as 
$1$ in $\mathcal L$ plays a role of such an element. Routine check shows that such an
element exists only in $\mathfrak g_{4, 8}$ (with parameter -1), and in $\mathfrak g_{4, 9}$ with a parameter 0. Finally, observing the multiplication table of $\mathfrak g_{4, 9}$ (with zero parameter), we notice that it is identical to the multiplication of $\mathcal L$ with respect to a new basis given by
\[1' = \dfrac{1}{2}1, i'=\dfrac{1}{2} i, v'=\dfrac{1}{\sqrt{2\alpha}} v, w' = \dfrac{1}{\sqrt{2\alpha}} w, \text{ \quad if }\alpha >0. \] If $\alpha < 0$, then we use the following change of basis: 
\[1' =- \dfrac{1}{2}1, i'=\dfrac{1}{2} i, v'=\dfrac{1}{\sqrt{-2\alpha}} v, w' = \dfrac{1}{\sqrt{-2\alpha}} w.\] This implies that $\mathcal L \cong \mathfrak g_{4,9}$ (with zero parameter).
The proof is complete.
\end{proof}

%%%%%%%%%%%%%%%%%%%%%%%%%%%%%%%%%%%%%%%%%%%%%%%%%%%%%%
\subsection*{Acknowledgments}

The authors express their sincere gratitude to the referees for their insightful comments and constructive suggestions, which have greatly enhanced the quality and clarity of the final version of this paper.

%%% REFERENCES %%%

\EditInfo{January 14, 2025}{April 4, 2025}{Ivan Kaygorodov}


{\small
\begin{thebibliography}{10}

\bibitem{AHK}
S.~Althoen, K.~Hansen, and L.~Kugler.
\newblock Four-dimensional real algebras satisfying middle {$\mathbb C$}-associativity.
\newblock {\em Algebras Groups Geom.}, 17(2):123--148, 2000.

\bibitem{AHK04}
S.~C. Althoen, K.~D. Hansen, and L.~D. Kugler.
\newblock A survey of four-dimensional {$\mathbb C$}-associative algebras.
\newblock {\em Algebras Groups Geom.}, 21(1):9--27, 2004.

\bibitem{AK}
S.~C. Althoen and L.~D. Kugler.
\newblock When is {${\mathbb R}\sp{2}$} a division algebra?
\newblock {\em Amer. Math. Monthly}, 90(9):625--635, 1983.

\bibitem{BO}
G.~M. Benkart and J.~M. Osborn.
\newblock Derivations and automorphisms of nonassociative matrix algebras.
\newblock {\em Trans. Amer. Math. Soc.}, 263(2):411--430, 1981.

\bibitem{CKMMRR}
A.~Calder\'{o}n, A.~Kaidi, C.~Mart\'{\i}n, A.~Morales, M.~Ram\'{\i}rez, and
  A.~Rochdi.
\newblock Finite-dimensional absolute-valued algebras.
\newblock {\em Israel J. Math.}, 184:193--220, 2011.

\bibitem{CL}
A.~Chapman and I.~Levin.
\newblock Alternating roots of polynomials over {C}ayley-{D}ickson algebras.
\newblock {\em Commun. Math.}, 32(2):63--70, 2024.

\bibitem{CV}
A.~Chapman and S.~Vishkautsan.
\newblock Roots and dynamics of octonion polynomials.
\newblock {\em Commun. Math.}, 30(2):25--36, 2022.

\bibitem{ED2}
E.~Darp\"{o} and A.~Rochdi.
\newblock Classification of the four-dimensional power-commutative real
  division algebras.
\newblock {\em Proc. Roy. Soc. Edinburgh Sect. A}, 141(6):1207--1223, 2011.

\bibitem{DDR}
A.~Diabang, O.~Diankha, and A.~Rochdi.
\newblock On the automorphisms of absolute-valued algebras.
\newblock {\em Int. J. Alg.}, 10(3):113--123, 2016.

\bibitem{D}
E.~Dieterich.
\newblock Classification, automorphism groups and categorical structure of the
  two-dimensional real division algebras.
\newblock {\em J. Algebra Appl.}, 4(5):517--538, 2005.

\bibitem{HKR}
F.~Hirzebruch, M.~Koecher, and R.~Remmert.
\newblock {\em Numbers}.
\newblock Graduate texts in Mathematics (vol.123). Springer-Verlag, New York,
  1990.

\bibitem{HP}
M.~H\"{u}bner and H.~P. Petersson.
\newblock Two-dimensional real division algebras revisited.
\newblock {\em Beitr\"{a}ge Algebra Geom.}, 45(1):29--36, 2004.

\bibitem{BK}
E.~Kleinfeld.
\newblock A characterization of the {C}ayley numbers.
\newblock {\em Studies in Modern Algebra}, 2:126–143, 1963.

\bibitem{Mer}
D.~M. Merriell.
\newblock Flexible almost alternative algebras.
\newblock {\em Proc. Amer. Math. Soc.}, 8:146--150, 1957.

\bibitem{Mub}
G.~Mubarakzyanov.
\newblock On solvable {L}ie algebras.
\newblock {\em Izv. Vys. Ucheb. Zaved. Matematika}, 1:114--123, 1963.

\bibitem{Nik}
A.~A. Nikitin.
\newblock Almost alternative algebras.
\newblock {\em Algebra i Logika}, 13(5):501--533, 605, 1974.

\bibitem{Rod}
A.~R. Palacios.
\newblock Absolute-valued algebras, and absolute-valuable {B}anach spaces.
\newblock In {\em Advanced courses of mathematical analysis {I}}, pages
  99--155. World Sci. Publ., Hackensack, NJ, 2004.

\bibitem{Ram}
M.~Ram\'{\i}rez~\'{A}lvarez.
\newblock On four-dimensional absolute-valued algebras.
\newblock In {\em Proceedings of the {I}nternational {C}onference on {J}ordan
  {S}tructures ({M}\'{a}laga, 1997)}, pages 169--173. Univ. M\'{a}laga,
  M\'{a}laga, 1999.

\bibitem{Sh}
I.~P. Shestakov.
\newblock Certain classes of noncommutative {J}ordan rings.
\newblock {\em Algebra i Logika}, 10, 1971.
\end{thebibliography}
}
\end{document}